\documentclass[11pt]{amsart}
\usepackage{latexsym}
\usepackage{amssymb}
\usepackage{amsxtra}
\usepackage{amsmath}
\usepackage{amsfonts}

\newcommand {\sectionnew}[1]{\section{#1}}
\newcommand{\beq}{\begin{equation}}
\newcommand{\eeq}{\end{equation}}
\newcommand{\beqa}{\begin{eqnarray}}
\newcommand{\eeqa}{\end{eqnarray}}
\newcommand{\beaa}{\begin{eqnarray*}}
\newcommand{\ben}{\begin{eqnarray*}}
\newcommand{\eaa}{\end{eqnarray*}}
\newcommand{\een}{\end{eqnarray*}}

\newcommand \nc {\newcommand}

\newtheorem{theorem}{Theorem}
\newtheorem{lemma}[theorem]{Lemma}

\newtheorem{corollary}[theorem]{Corollary}

\nc \thref{Theorem \ref}
\nc \leref{Lemma \ref}
\nc \prref{Proposition \ref}
\nc \coref{Corollary \ref}
\nc \deref{Definition \ref}
\nc \exref{Example \ref}
\nc \reref{Remark \ref}

\newcommand{\C}{\mathbb{C}}

\renewcommand{\H}{\mathcal{H}}

\newcommand{\T}{\mathcal{T}}

\newcommand{\Z}{\mathbb{Z}}

\newcommand{\f}{\mathbf{f}}
\newcommand{\g}{\mathbf{g}}

\def\deg{\mathop{\rm deg}\nolimits}
\def\dim{\mathop{\rm dim}\nolimits}

\def\d{\partial}

\def\iso{\cong}
\def\tensor{\otimes}

\def\({\left(}
\def\){\right)}
\def\<{\left\langle}
\def\>{\right\rangle}

\def\lieh{{\mathfrak{h}}}
\def\lieg{{\mathfrak{g}}}

\def\gl{\lambda}
\def\ge{\epsilon}
\def\ga{\alpha}
\def\gd{\delta}
\def\gb{\beta}

\nc{\on}{\operatorname}
\nc{\ol}{\overline}

\begin{document}
\title{Soliton equations, vertex operators, and simple singularities}

\author[E. Frenkel]{E. Frenkel$^1$}\thanks{$^1$ Supported by DARPA
  through the grant HR0011-09-1-0015}

\address{Department of Mathematics, University of California,
Berkeley, CA 94720, USA}

\author[A. Givental]{A. Givental$^2$}\thanks{$^2$ Supported by the NSF
grant DMS-0604705}

\address{Department of Mathematics, University of California,
Berkeley, CA 94720, USA}

\author[T. Milanov]{T. Milanov$^3$}\thanks{$^3$ Supported by the NSF
grant DMS-0707150}

\address{Department of Mathematics, University of Michigan,
Ann Arbor, MI 48109, USA}

\date{September 2009}

\begin{abstract}
We prove the equivalence of two hierarchies of soliton equations
associated to a simply-laced finite Dynkin diagram. The first was
defined by Kac and Wakimoto \cite{KW} using the principal realization
of the basic representations of the corresponding affine Kac--Moody
algebra. The second was defined in \cite{GM} using the Frobenius
structure on the local ring of the corresponding simple
singularity. We also obtain a deformation of the principal realization
of the basic representation over the space of miniversal deformations
of the corresponding singularity. As a by-product, we compute the
operator product expansions of pairs of vertex operators defined in
terms of Picard--Lefschetz periods for more general singularities.
Thus, we establish a surprising link between twisted vertex operators
and deformation theory of singularities.
\end{abstract}

\maketitle

\sectionnew{Introduction}

The principal hierarchy of soliton equations associated to an affine
Kac--Moody algebra of type $X^{(1)}_N$, where $X=ADE$, has been
defined by V. Kac and M. Wakimoto \cite{KW} as the following systems
of Hirota bilinear equations: \beqa
\label{hqe_kw} && {\rm Res}\, \frac{d\zeta}{\zeta} \Big( \sum_{i=1}^N
a_i \Gamma^{\ga_i}\tensor\Gamma^{-\ga_i} \Big)\, (\tau\tensor\tau) =
h^{-2}(\rho | \rho)(\tau\tensor\tau) + \\ \notag && h^{-1}\Big(
\sum_{m\in E_+} m(y_m\tensor 1-1\tensor y_m)(\d_{y_m}\tensor 1-
1\tensor \d_{y_m})\Big) (\tau\tensor\tau). \eeqa Here
$\Gamma^{\pm\ga_i}$ are vertex operators
\beq\label{vop_kw} \Gamma^{\pm\ga_i} = \exp \Big(\pm\sum_{m\in E_+}
\gb_{i,m}\,y_m\, \zeta^m\Big) \exp \Big(\mp\sum_{m\in E_+}
\gb_{i,-m}\,\d_{y_m}\, \zeta^{-m}/m\Big),
\eeq acting on a certain Fock space
$\C[y_m, m\in E_+]$. The Fock space and the operators are constructed from 
the following data associated to a root system $X_N$ of $ADE$ type: 
\begin{itemize}
\item   
$M$ is a {\em Coxeter transformation} of the root system.
\item $\alpha_i$, $1\leq i\leq N$, are roots chosen one
from each orbit of $M$ on the set of roots. 
\item $h$ is the Coxeter number (i.e., the order of $M$). 
\item $m_a$ are the Coxeter exponents (in particular, this means that 
$e^{2\pi i m_a/h}$, $1\leq a\leq N$, are the eigenvalues of $M$),
  ordered, so that $m_a\leq m_b$ when $a<b$.
\item $E_+=\{ (a,n)\ |\ 1\leq a \leq N,\ n\in \Z_{\geq 0}\}$. 
Abusing notation, we
will write $m\in E_{+}$ for $m=m_a+nh$. (In all cases except 
$X_N=D_{N}$ with $N$ even this notation is unambiguous since it embeds 
$E_{+}$ as a subset into $\Z_{>0}$.) 
\item $\gb_{i,m}=\ga_i(H_{a(m)}),$ where the subscript $a(m)$ is
defined by $m=m_{a(m)}+nh$, and $\{H_{a}\ |\ a=1,\dots,N\}$ is an
eigenbasis of $M$ satisfying the normalization condition
$(H_{a},H_{b})=h \delta_{a+b,N+1}$. Here $(\cdot\ |\ \cdot)$ is the
invariant inner product normalized by the condition $(\ga\ |\ \ga)=2$
for all roots $\ga$.
\item $(\rho\ |\ \rho)=Nh(h+1)/12$ is the value of the inner
square of the sum $\rho$ of fundamental weights (see, e.g.,
\cite{Kac}).
\end{itemize}

The coefficients $a_i$ are defined in terms of the principal
realization of the basic representation (see \cite{KW} and Section
\ref{basic} below).

\medskip

On the other hand, an {\em a priori} different hierarchy was
constructed in \cite{GM} for each simple singularity of $ADE$ type
using the Picard--Lefschetz periods of the singularity. Moreover, it
was shown in \cite{GM} that the equations of this hierarchy have the
same form as (\ref{hqe_kw}) (for the corresponding root system of
$ADE$ type) except that the coefficients, which we will now denote
$\widetilde{a}_i$ (instead of $a_i$) are {\em a priori}
different. These coefficients, which are actually defined in terms of
certain limits associated with the singularity, were characterized in
\cite{GM} by the following conditions:
\begin{equation}    \label{ratios}
\widetilde{a_i}/\widetilde{a_j} = \left. \prod_{\ga} 
(H_{1}\ |\, \ga)^{-(\ga_i\, |\, \ga)^2/2}\ \right/ \ \prod_{\ga} (H_{1}\
|\, \ga)^{-(\ga_j\, |\, \ga)^2/2},\ \ \text{and}\ \sum_{i=1}^N
\widetilde{a}_i = h^{-2}(\rho\ |\ \rho).
\end{equation}
It was conjectured in \cite{GM} that the two hierarchies coincide (in
other words, $\widetilde{a}_i=a_i$ for all $i$), and the conjecture
was partially verified in \cite{GM} and \cite{Wu}. However, it was
left open in general; one of the problems was that the Kac--Wakimoto
coefficients $a_i$ were not given in all cases.

In this paper we prove the following result for all $ADE$ types:

\bigskip

{\bf Theorem.} {\em The two hierarchies coincide; namely, 
\[ a_i=\widetilde{a}_i\ \ \text{for all $i=1,\dots, N$.}\]}

Our proof is based on computing explicitly the operator product
expansions (OPEs) for the vertex operators defined in \cite{GM} in the
context of singularity theory and matching them with the OPEs of the
vertex operators of representation theory. We then use these OPEs to
prove that $a_i=\widetilde{a}_i$. Actually, the vertex operators
defined from singularity theory and their OPEs make sense for
singularities much more general than the $ADE$ singularities, and in a
sense the main goal of this paper is to expose their place in
representation theory.

\medskip

In fact, it turns out that a similar result holds in the case of more
general singularities. In this case the role of an affine Kac--Moody
algebra is played by the lattice vertex algebra associated to the
Milnor lattice of the singularity (the middle cohomology of the
generic fiber of the singularity equipped with the intersection
pairing). Using the Picard--Lefschetz periods associated to the
singularity, we define twisted vertex operators which give rise to
twisted modules over this lattice vertex algebra (with respect to the
automorphism of the lattice given by the monodromy) in a similar
fashion to what is described in this paper in the case of simple
singularities. This will be discussed in more detail in a subsequent
paper.

\medskip

Thus, we prove that the Kac--Wakimoto hierarchy, which is constructed
using the basic representation of the affine Kac--Moody algebra
$X^{(1)}_N$, where $X=ADE$, and the hierarchy of \cite{GM}, which is
attached to the singularity of type $X_N$, coincide. However, the
latter is naturally included in a family of hierarchies parametrized
by the space of miniversal deformations of the singularity, using the
Frobenius structure of the singularity. In Section \ref{def} we will
use this structure to produce a {\em family} of representations of
$X^{(1)}_N$ defined by certain deformations
$\Gamma^\alpha_\tau(\zeta)$ of the vertex operators
$\Gamma^\alpha(\zeta)$ introduced above, where $\tau$ lies in the
space of miniversal deformations of the singularity of type
$X_N$. These deformations are given in terms of the Picard--Lefschetz
periods associated to simple singularities. Thus, we obtain a
surprising link between representations of affine Kac--Moody algebras
of $ADE$ type and simple singularities.

\medskip

Although this is not used in our proof of the above theorem, we also
obtain an explicit formula for the coefficients $a_i$ for all
simply-laced affine Kac--Moody algebras: \beq\label{ai} a_i =h^{-1}
\prod_{k=1}^{h-1} \(1-\eta^k\)^{(\ga_i\, |\,M^k\ga_i)}, \qquad \qquad
\eta=e^{2\pi\sqrt{-1}/h}.\eeq Formulas \eqref{ratios}, \eqref{ai} and
the theorem then imply that the following ratios are the same for all
roots $\ga_i$:
\[ \left. \prod_{k=1}^{h-1} \(1-\eta^k\)^{(\ga_i\, |\,M^k\ga_i)}\
\right/ \ \prod_{\ga} \; (H_{1}\ |\, \ga)^{-(\ga_i\, |\, \ga)^2/2}. \]
While it should not be hard to verify this on the case-by-case basis,
it is an interesting question whether one can find a general direct
proof of this fact, which expresses the proportionality between $\{
a_i \}$ and $\{ \widetilde{a}_i \}$.  Since $\sum a_i$ (which can be
easily found from the consistency condition on the Kac--Wakimoto
hierarchy) coincides with $\sum \widetilde{a}_i$, this would provide
another, more elementary, proof of our theorem.

\medskip

The paper is organized as follows. In Section 2, we recall the
construction of hierarchies of Kac--Wakimoto and derive formula
(\ref{ai}).

In Section 3, we recall from \cite{GM} the construction of vertex
operators based on Frobenius structures of singularity theory, and in
the case of simple singularities, provide a uniform identification of
the coefficients $\widetilde{a}_i$ with their counterparts in the
Kac--Wakimoto theory for all $ADE$ types.

In Section 3, we already use {\em families} of vertex operators
parametrized by the miniversal deformation of a (simple)
singularity. In Section 4, we show that they define a family of
realizations of the basic representation of level 1 of the
corresponding affine Kac--Moody algebra. In fact, the intertwiners
between the realizations of this family have already been constructed
in \cite{GM}.

\sectionnew{The Kac--Wakimoto hierarchy}
\label{basic}

In abstract terms, the Kac--Wakimoto hierarchy describes the points of
the Grassmannian which is the orbit of the highest weight vector in
the level 1 basic representation (more precisely, in its principal
realization) of the affine Kac--Moody algebra under the action of the
corresponding group.

Let $\lieg$ be a simple Lie algebra over $\C$ of the type $X_N=A_N,
D_N$, or $E_N$. By definition, the affine Kac--Moody algebra
corresponding to $\lieg$ is the vector space $$\widehat{\lieg}:=\lieg
[t,t^{-1}] \oplus \C\,K \oplus \C\, d$$ equipped with the Lie bracket
defined by the following relations: \ben && [X\ t^n,Y\ t^m]:=[X,Y]\
t^{n+m} + n\delta_{n,-m}(X\, |\, Y)K,\\ && [d,X\ t^n]:= n(X\
t^n),\quad [K,\widehat{\lieg}]:=0.  \een Here $X,Y\in \lieg$, and
$(\cdot\ |\ \cdot)$ denotes the adjoint-invariant bilinear form on
$\lieg$ normalized so that $(\ga\, |\, \ga)=2$ for all roots $\alpha$.

The principal realization of the level 1 basic representation depends
on the choice of a Coxeter element $M$. Let $\lieh$ be a Cartan
subalgebra in $\lieg$, and $$\lieg=\lieh\oplus \bigoplus_{\ga \in
\Delta} \lieg_{\ga}$$ its root decomposition. Extend the action of $M$
on $\lieh$ to a finite order inner automorphism of $\lieg$ in the
standard way. Under the action of $M$ on $\lieh$ all roots form $N$
orbits of cardinality $h$, and the corresponding root spaces are
likewise permuted by $M$: $M \cdot \lieg_{\ga}=\lieg_{M\ga}$. We pick
generators $A_{\ga}\in \lieg_{\ga}$ in such a way that they form an
$M$-invariant set of vectors. We fix representatives $\ga_1, \dots,
\ga_N$ in the $M$-orbits on the root system.  Put $\eta=e^{2\pi
i/h}$. Using this notation, we lift the action of $M$ to
$\widehat{\lieg}$:
\[  M (X\ t^m) = MX \ (\eta^{-1} t)^m,\ \ \ MK=K,\ \ Md=d.\]
We will frequently use the projector $x\mapsto h^{-1}\sum_{k=1}^h M^k
x$ to the Lie subalgebra $\widehat{\lieg}^M$ of $M$-invariant vectors.
It is an important, and non-trivial, fact (see \cite{Kac} for a proof)
that {\em the subalgebra $\widehat{\lieg}^M$ is isomorphic to the Lie
algebra $\widehat{\lieg}$ itself}.  The principal construction of the
basic representation is based on the property of $M$-invariant space
\[ \lieh [t,t^{-1}]^M\oplus \C K \]
to be a Lie subalgebra isomorphic to the Heisenberg Lie algebra. An
important result of \cite{KKLW} (see also \cite{Kac}) is that
{\em the standard level 1 Fock representation of this Heisenberg
algebra extends uniquely to a representation of $\widehat{\lieg}^M$.}

A vector $v$ in the orbit of the highest weight vector in this
representation under the action of the corresponding group may be
characterized by the property that $v\otimes v$ is an eigenvector,
with a certain specific eigenvalue, of the bilinear Casimir
operator. Introduce a basis $H_1,\dots, H_N$ of $\lieh$ formed by
eigenvectors of $M$ with the eigenvalues $\eta^{m_i}$, ordered so that
$m_i\leq m_j$ when $i<j$, and normalized by the condition $$(H_i |
H_j)=h \gd_{i+j,N+1}.$$ Introduce the notation
\[ H_{i,m}= h^{-1} \sum_{k=1}^h M^k (H_i\ t^m),\ \ \text{and}\ \  
A_{\pm \ga_i,m}= h^{-1} \sum_{k=1}^h M^k (A_{\pm \ga_i}\ t^m), \ \
i=1,\dots, N. \] The elements $K, d, A_{\ga_i,m}$ with $m\in \Z$, and
$H_{i,m}$ with $m\equiv m_i\mod h$ (note that all other $H_{i,m}=0$)
form a basis in the Lie algebra $\widehat{\lieg}^M$.  Restricting an
invariant inner product from $\widehat {\lieg}$ to $\widehat{\lieg}^M$
and computing it in this basis, it is not hard to see that the
following element of $\widehat{\lieg}^M\otimes \widehat{\lieg}^M$
commutes with the diagonal action of $\widehat{\lieg}^M$:
\[ \sum_{i,m}\frac{h}{(A_{\ga_i}\, |\, A_{-\ga_i})}
A_{\ga_i,m}\otimes A_{-\ga_i,-m} + \frac{1}{h}\!
\sum_{\substack{i+j=N+1\\ m\equiv m_i\mod h}} H_{i,m} \otimes
H_{j,-m}\ +\ K\otimes d+d\otimes K.\] This is the bilinear Casimir
element.

Consider the representation of the Heisenberg Lie subalgebra $\lieh
[t,t^{-1}]^M\oplus \C K$ on the Fock space $\C [y_m\, |\, m\in E_{+}]$
given by the formulas \[ K\mapsto 1/h, \ \ \text{and}\ \ H_{i,\pm m}
\mapsto \left\{
\begin{array}{l}\partial/ \partial y_m, \\ m y_m \end{array} \right. \
\text{for $m\in E_{+}$}.\]

According to \cite{KKLW}, this representation of the Heisenberg Lie
subalgebra on the Fock representation extends to the level 1 basic
representation of $\widehat{\lieg}^M$. From the commutation relations
of $d$ with $H_{i,m}$, it follows immediately that under this
action\footnote{Up to an additive constant, which can be made $0$ by
redefining $d$ as $d-const\cdot K$.}
\[ d\mapsto -\sum m y_m \partial/\partial y_m, \]
the Euler vector field for the grading $\deg y_m=-m$ on the Fock space.

Following \cite{KW}, we use generating functions
\[ x_{\pm \ga_i}(\zeta) :=\sum_m A_{\pm \ga_i,m} \zeta^{-m}.\]
It is easy to check that when $m\equiv m_i \mod h$, 
\[ [ H_{i,m}, x_{\pm \ga_j}(\zeta) ] = \pm \ga_j(H_i) \zeta^{-m} 
x_{\pm \ga_j}(\zeta).\]
This commutation relation coincides with the one between the operators 
representing $H_{i,m}$ in the Fock space and the vertex operators
\beq \label{vop_kw_again} 
\Gamma^{\pm\ga_i} = \exp \Big(\pm\sum_{m\in E_+}
\ga_i(H_{a(m)})\,y_m\, \zeta^m\Big) \exp \Big(\mp\sum_{m\in E_+}
\ga_i(H_{a(-m)})\,\d_{y_m}\, \zeta^{-m}/m\Big).
\eeq
Moreover, according to a general lemma from \cite{Kac}, this implies
that $x_{\pm \ga_i}$ is proportional to $\Gamma^{\pm \ga_i}$. Putting
$\deg \zeta =-1$, we make both the vertex operators and the generating
functions homogeneous of degree $0$ with respect to the grading in the
Fock space defined by $d$. This implies that the coefficient of
proportionality does not depend on $\zeta$.

Define $a_i$ by the formula
\beq\label{prop_coef}
\frac{h}{(A_{\ga_i}\, |\, A_{-\ga_i})} 
 x_{\ga_i}(\zeta)\otimes x_{-\ga_i}(\zeta) = a_i \Gamma^{\ga_i}(\zeta)
 \otimes \Gamma^{-\ga_i}(\zeta).
\eeq
Noting that
\[ \operatorname{Res}\ \frac{d\zeta}{\zeta} x_{\ga_i}(\zeta)\otimes 
x_{-\ga_i}(\zeta)=\sum_m A_{\ga_i,m}\otimes A_{-\ga_i,-m},\]
we express the bilinear Casimir operator in our representation as
follows:
\[ \operatorname{Res} 
\frac{d\zeta}{\zeta} \sum_i a_i \Gamma^{\ga_i}(\zeta) \otimes
\Gamma^{-\ga_i}(\zeta)\ - \ \sum_{m\in E_{+}} \frac{m}{h} (y_m\otimes
1-1\otimes y_m) (\partial_{y_m}\otimes 1 - 1\otimes \partial_{y_m}).\]
The eigenvalue of this operator on tensor squares $v\otimes v$, where
$v$ is the orbit of the highest weight vector under the action of the
corresponding group, may be computed using an explicit (and
non-trivial) identification $\widehat{\lieg}\to \widehat{\lieg}^M$.
We are not going to reproduce this computation here, but merely quote
the answer (see \cite{Kac}): $(\rho\, |\, \rho)/h^2$. Denoting by
$\tau$ the element of the Fock representation corresponding to $v$, we
arrive at the Hirota bilinear equation (\ref{hqe_kw}) from the
Introduction.

We now wish to compute the coefficients $a_i$ appearing in formula
\eqref{prop_coef}.

\medskip

\begin{lemma}\label{kw_ope}
The expression
\[ (\zeta^h-w^h)\left( 
\frac{x_{\ga_i}(\zeta)}{\zeta} \frac{x_{-\ga_i}(w)}{w} - 
\frac{(A_{\ga_i}\, |\, A_{-\ga_i})}{h} \frac{K}{(\zeta-w)^2}\right) , \]
expanded in the region $|w|<|\zeta|$ into a formal power series 
in $\zeta^{\pm 1}$ and $w^{\pm 1}$ with coefficients which are operators 
in our Fock space, has a well-defined limit as $\zeta \to w$.
\end{lemma}

\medskip

\proof  

Introduce the {\em normal ordering} by the formula
\begin{equation}    \label{no}
:\! A_{\ga_i,n}A_{-\ga_i,l} \! : =
\begin{cases} A_{\ga_i,n}A_{-\ga_i,l} & \mbox{ if } n< 0 \\ 
A_{-\ga_i,l}A_{\ga_i,n} &
\mbox{ if } n\geq 0.
\end{cases}
\end{equation}
Then the formal power series $: \! x_{\ga_i}(\zeta) x_{-\ga_i}(w) \!
:$ has a well-defined limit as $\zeta \to w$ as an operator acting on
individual elements of the Fock space. We compute the difference
$$
x_{\ga_i}(\zeta) x_{-\ga_i}(w)-: \! x_{\ga_i}(\zeta) x_{-\ga_i}(w) \!
: = \sum_{m\in \Z}\sum_{n\geq 0} [A_{\ga_i,n},A_{-\ga_i,m}]
\zeta^{-n}w^{-m}.
$$
We have
\[ [A_{\ga_i,n},A_{-\ga_i,m}] = \frac{(A_{\ga_i}\, |\, A_{-\ga_i})}{h} n 
\delta_{m+n,0} K \ +\ h^{-2}\sum_{k,l=1}^h [M^lA{\ga_i},M^kA_{-\ga_i}]
\eta^{-nl-mk}t^{n+m}.\]
Denoting $[M^l A_{\ga_i},A_{-\ga_i}]$ by $B_i^l$, we can rearrange 
contributions of second summands as
\[ h^{-2}\sum_{m\in \Z}\sum_{n\geq 0} \zeta^{-n}w^{-m} 
\sum_{l=1}^h \eta^{-ln} \sum_{k=1}^{h} M^k B_i^l \eta^{-k(n+m)} t^{n+m} =
\frac{1}{h}\sum_{l=1}^h 
\frac{\eta^l \zeta}{\eta^l\zeta-w} \sum_{m\in \Z}B^l_{i,m}w^{-m}. 
\]
The summands have pole at $w=\eta^l\zeta$ of order at most $1$, and 
hence have a well-defined limit after multiplication by $\zeta^h-w^h$. 
The first summands of the commutators add up to
\[ \frac{(A_{\ga_i}\, |\, A_{-\ga_i})}{h}
\sum_{n\geq 0} n \left(\frac{w}{\zeta}\right)^n K = \frac{\zeta
 w}{(\zeta-w)^2} \frac{(A_{\ga_i}\, |\, A_{-\ga_i})}{h} K.\]      
\qed

\medskip

The multiplication and differentiation parts of the vertex operators
(\ref{vop_kw_again}) are elements of a Heisenberg group. Such
operators commute up to a scalar factor.  Define $B_i(\zeta, w)$ by
the formula (OPE) \beq\label{ope_gamma1}
\Gamma^{\ga_i}(\zeta)\Gamma^{-\ga_i}(w) = B_i(\zeta,w) : \!
\Gamma^{\ga_i}(\zeta)\Gamma^{-\ga_i}(w) \! :\ , \eeq where the normal
ordering on the RHS is defined by moving all differentiation operators
$\d/\d y_m$ to the right (i.e., by taking the commutator of the
differentiation part of the vertex operator on the left with the
multiplication part of the vertex operator on the
right).\footnote{Note that this normal ordering is different from the
one defined by formula \eqref{no}.} The normally ordered product on
the RHS has a well-defined limit as $\zeta \to w$, and moreover, this
limit is obviously equal to $1$.

\medskip

\begin{corollary}\label{ai_limit1}
\[ 
a_i^{-1} = \lim_{\zeta\rightarrow w} (1-w/\zeta)(1-w^h/\zeta^h)
B_i(\zeta,w), \]
\end{corollary}

\medskip

\proof We have
\[ (\zeta-w)(\zeta^h-w^h)\Gamma^{\ga_i}(\zeta)\Gamma^{-\ga_i}(w) =
 (\zeta-w)(\zeta^h-w^h)
 B_i(\zeta,w):\Gamma^{\ga_i}(\zeta)\Gamma^{-\ga_i}(w):\]
\[ =\frac{h a_i^{-1} (\zeta-w)(\zeta^h-w^h)}{(A_{\ga_i}\, |\, A_{-\ga_i})} 
 x_{\ga_i}(\zeta) x_{-\ga_i}(w) = a_i^{-1} \zeta w
 \frac{\zeta^h-w^h}{\zeta-w} K + (\zeta-w) \text{(regular terms)}. \]
Passing to the limit $\zeta \to w$ and using Lemma \ref{kw_ope} and
the fact that $K\mapsto 1/h$ in our representation, we obtain the
desired result. \qed

\medskip

\begin{lemma}\label{B_i}
\beq\label{B_factor2} B_i(\zeta,w) = \prod_{k=1}^{h}
\Big(1-\eta^k\,\frac{w}{\zeta}\Big)^{-(\ga_i\, |\, M^k\ga_i)}, \eeq
\end{lemma}
\proof
The factor $B_i(\zeta,w)$ in \eqref{ope_gamma1} is obtained by
commuting the second and the first exponential factors, respectively,
of $\Gamma^{\ga_i}$ and $\Gamma^{-\ga_i},$ i.e., \beq\label{B_factor}
B_i(\zeta,w) = \exp \Big( \sum_{m\in E_+}
\ga_i(H_{a(-m)})\ga_i(H_{a(m)}) \frac{(w/\zeta)^m}{m} \Big).
\eeq 
Consider the projection $h^{-1}\sum_{k=1}^h \eta^{mk} M^k\ga_i$ of
$\ga_i$ to the eigenspace of $M$ in $\lieh^*$ with the eigenvalue
$\eta^{-m}$.  Since the eigenspaces are pairwise orthogonal and the
bases $\{ H_a \}$ and $\{ h^{-1}H_{N+1-a} \}$ formed by the
eigenvectors are dual, the projection of $\ga_i$ can be written as
$h^{-1}(\ga_i\, |\, H_{a(m)}) H_{a(-m)}$.  Pairing this with $\ga_i$,
we find that \ben \ga_i(H_{a(-m)})\ga_i(H_{a(m)})= \sum_{k=1}^{h}
\eta^{mk} (\ga_i\, |\, M^k\ga_i), \een Note that this identity extends
to $m\notin E_{+}$, since the RHS is equal to zero in this
case. Therefore
\[ \ln B_i = \sum_{k=1}^h (\ga_i\, |\, M^k\ga_i) \sum_{m=1}^{\infty}
\frac{(\eta^k w/\zeta)^m}{m}=-\sum_{k=1}^h (\ga_i\, |\, M^k\ga_i) 
\ln (1-\eta^k\frac{w}{\zeta}).\] 
\qed

Combining Lemma \ref{B_i} and Corollary \ref{ai_limit1}, we obtain

\begin{corollary}    \label{t3}
\beq\label{ai1}
a_i = h^{-1} \prod_{k=1}^{h-1} \(1-\eta^k\)^{(\ga_i\, |\, M^k\ga_i)}.
\eeq
\end{corollary}

\medskip

The first equation of the Kac--Wakimoto hierarchy reads
$$
\sum_{i=1}^N a_i = h^{-2} (\rho | \rho),
$$
which coincides with the value of $\sum_{i=1}^N \widetilde{a}_i$ (see
formula \eqref{ratios}). Hence in order to prove that
$a_i=\widetilde{a}_i$ it suffices to check that
$$
\frac{a_i}{a_j} = \frac{\widetilde{a}_i}{\widetilde{a}_j},
$$
where the right hand side is given by formula \eqref{ratios}. This has
been done for $A_N$, $D_4$ and $E_6$ in \cite{GM} and for $D_N$ in
\cite{Wu}. It would be interesting to obtain a direct uniform proof of
this equality which is not based on case-by-case calculations. It would
also be interesting to understand the meaning of the right hand side
of formula \eqref{ratios} with $H_1$ replaced by $H_m$.

\medskip

\noindent{\em Remark.}  The results of this section may be interpreted
in the context of the theory of twisted modules over vertex
algebras. Consider the lattice vertex algebra $V_Q$ corresponding to
the root lattice $Q$ of type $X_N$, where $X=ADE$. This vertex algebra
is isomorphic to the basic representation of the affine Kac--Moody
algebra $\widehat{\lieg}$ of type $X^{(1)}_N$ in the homogeneous
realization, viewed as a vertex algebra (see, e.g.,
\cite{Kac1}). Hence $\sigma$-twisted modules over $V_Q$, where
$\sigma$ is an automorphism of $\lieg$ preserving the corresponding
Cartan subalgebra, realize modules over the twisted affine Kac--Moody
algebra $\widehat\lieg_\sigma$ \cite{BK}. In particular, taking the
Coxeter transformation $M$ as the automorphism $\sigma$, we obtain the
basic representation of $\widehat\lieg_\sigma$ in the principal
realization. The corresponding twisted operators are equal to
$\Gamma^{\pm \ga_i}(\zeta)$ up to scalar multiples. The OPE between
the twisted vertex operators may be found from the OPE between the
corresponding untwisted vertex operators, and this gives a way to
compute these scalar multiples (see \cite{BK,L}).

This observation allows us to generalize the results of this paper to
the case of more general singularities. In this case the role of an
affine Kac--Moody algebra is played by the lattice vertex algebra
associated to the Milnor lattice of the singularity (the middle
cohomology of the generic fiber of the singularity equipped with the
intersection pairing). Using the Picard--Lefschetz periods associated
to the singularity, we define twisted vertex operators which realize
twisted modules over this lattice vertex algebra. In the next two
sections we describe this in the case of simple singularities, leaving
the general case for a subsequent paper.\qed

\sectionnew{Vertex operators from singularities}    \label{sing}

In this section we recall the setup of \cite{GM} and describe the
integrable hierarchies associated to simple singularities. We will
then identify them with the Kac--Wakimoto hierarchies in a uniform way
for all $ADE$ types.

\subsection{Periods associated to isolated critical points}

Suppose that we are given a polynomial $f:\C^{2l+1}\rightarrow \C$
which has an isolated singularity of multiplicity $N$ at the
origin. Denote by $H$ the local algebra
$\C[[x_1,\ldots,x_{2l+1}]]/(\d_{x_1}f,\ldots,\d_{x_{2l+1}}f)$ of the
singularity. We have $\dim H=N$.  Let the family $f_t, t\in \T\subset
\C^N$, be a miniversal deformation of $f$, i.e., $f_0=f$ and
$\partial_{t^a} f_t|_{t=0}, \ a=1,\dots, N$, represent a basis in $H$.
By picking a small ball $B_\rho^{2l+1}$ of dimension $2l+1$ in
$\C^{2l+1}$ centered at $0$, we can find sufficiently small disk
$B^{1}_\delta$ in $\C$ and ball $\T\subset\C^N$, so
that the fibers $f_t^{-1}(\gl)$, $(\gl,t)\in B_\delta^1\times \T$
intersect transversely the boundary of $B_\rho^{2l+1}$. We may, and
will, assume without loss of generality that the critical values of
$f_t$ are contained in a disk $B_{\delta_0}^1$ with radius
$\delta_0<1<\delta$.

Each tangent space $T_t\T$ is identified with the algebra of functions
on the critical scheme $Crit(f_t)$ by the map $\d_{t^a}\mapsto \d f/\d
t^a \ ({\rm mod }\ \d f_t/\d x_i) (1\leq a\leq N, 1 \leq i \leq
2l+1).$ The induced multiplication on $T_t\T$ is denoted by
$\bullet_t.$ Functions $f_t$ of the miniversal family restricted to
their critical schemes $Crit (f_t)\iso T_t\T$ define a vector field on
$\T$ denoted by $E$ and called the {\em Euler field}. Given a volume
form $\omega$ on $B_\rho^{2l+1}$, the following {\em residue pairing}
defines a non-degenerate bilinear pairing on $T\T$: \ben
(\d_{t^a},\d_{t^b})_t = \Big(\frac{1}{2\pi
i}\Big)^{2l+1}\int_{\Gamma_\ge} \frac{\partial_{t^a}f_t\
\partial_{t^b}f_t\ \omega} {\frac{\d f_t}{\d x_1}\ldots\frac{\d
f_t}{\d x_{2l+1}} } , \een where the integration cycle $\Gamma_\ge$ is
supported on $|\frac{\d f_t}{\d x_1}|= \ldots =|\frac{\d f_t}{\d
x_{2l+1}}|=\ge$.  According to K. Saito's theory of {\em primitive
forms} \cite{Sa}, there exists a volume form $\omega$ on
$B_{\rho}^{2l+1}$, possibly depending on $t\in \T$, such that the
residue pairing on $T\T$ is flat homogeneous (of certain degree) with
respect to the Euler vector field. Moreover, in the flat coordinate
system of the residue metric, the Gauss--Manin connections for various
period maps that one can associate to the miniversal deformation of a
singularity simultaneously assume a rather simple form. The resulting
datum defines on $\T$ a {\em conformal Frobenius structure} (in the
terminology of B. Dubrovin \cite{Db}). In the following paragraphs, we
introduce relevant notation. For the actual construction of the
Frobenius structures in singularity theory we refer the reader to
C. Hertling's book \cite{He}.

First, assuming that the form $\omega$ is fixed once and for all,
denote by $(\tau^1,\dots, \tau^N)$ a coordinate system on $\T$ which
is flat with respect to the residue metric, and write $\partial_a$ for
$\partial_{\tau^a}$. It follows from the homogeneity condition of the
metric that in a suitable flat coordinate system the Euler vector
field is the sum of a constant and linear vector fields:
\[ E=\sum (1-d_a) \tau^a \partial_{a} + \sum \rho_a \partial_a.\]
The constant part represents the class of $f$ in $H$, and the spectrum
of degrees $d_1,\dots, d_N$ ranges from $0$ to $\Delta$ (called the
{\em conformal dimension} of the Frobenius structure at hand) and
differs by a shift from the {\em Steenbrink spectrum} $\{ s_a \}$ of
the singularity, $s_a=d_a-1/2-\Delta/2+l$.

Let $V_{\gl,t} := f_{t}^{-1}(\gl) \cap B_{\rho}^{2l+1}$ denote the
Milnor fibers.  Choosing $(\gl,\tau)=(1,0)$ as a reference point in
$B_{\delta}^1\times \T$, pick a middle homology class $\varphi\in
H_{2l}(V_{(1,0)};\Z)\iso \Z^N$ in the Milnor lattice, and denote by
$\varphi_{\gl,t}$ its parallel transport (using the Gauss--Manin
connection) to the Milnor fiber $V_{\gl,t}$. Let $d^{-1}\omega$ mean
any $2l$-form whose differential is $\omega$. We can integrate
$d^{-1}\omega$ over $\varphi_{\gl,t}$ and obtain this way multivalued
functions of $\gl$ and $t$ ramified around the discriminant in
$B_{\delta}^1\times \T$ (over which Milnor fibers become singular). We
associate to $\varphi$ the following Picard--Lefschetz {\em period
vector} $I^{(k)}_\varphi(\gl,\tau)\in H\ (k\in \Z)$:
\beq\label{periods} (I^{(k)}_\varphi(\gl,\tau), \partial_a ):=
(2\pi)^{-l}\ (-\d_a)\ \d_\gl^{l+k}\ \int_{\varphi_{\gl,\tau}} \
d^{-1}\omega.  \eeq Note that this definition is consistent with the
operation of stabilization of singularities. Namely, adding the
squares of two new variables does not change the RHS since it is
offset by an extra differentiation $(2\pi)^{-1}\partial_{\gl}$. In
particular, this defines the period vector for negative values of $k$,
since we may take $k \geq -l$ in this formula with $l$ as large as
needed.
 
The period vectors \eqref{periods}, being flat sections of a
Gauss--Manin connection, satisfy a system of linear differential
equations (see \cite{He}, section 11) \beq\label{periods:de} \d_a
I^{(k)} = -(\d_a\bullet_t)\d_\gl I^{(k)},\ 1\leq a\leq N, \quad
(\gl-E\bullet_t) \d_{\gl}I^{(k)}= (\Theta-k-1/2) I^{(k)}, \eeq where
$\Theta: H\to H$ (called sometimes {\em Hodge grading operator} of the
Frobenius structure) is the operator anti-symmetric with respect to
the residue pairing and defined by
\[ \Theta (\partial_a) = \theta_a \partial_a, \ \
\theta_a=\frac{\Delta}{2}-d_a.\]
Using the last equation in (\ref{periods:de}) 
we analytically extend the period vectors to all 
$|\gl|>\delta$. It also follows from \eqref{periods:de} that the
period vectors have the following symmetry:
\ben
I^{(n)}_\varphi(\gl,\tau)\ = \ I^{(n)}_\varphi(0,\tau-\gl{\bf 1}),
\een
where $\tau \mapsto \tau-\gl{\bf 1}$ denotes the time-$\gl$
translation in the direction of the flat vector field ${\bf 1}$
obtained from $1\in H$.  (It represents the unit element for all the
products $\bullet_t$.)

\subsection{Vertex operators}

Let $\H:=H((z^{-1}))$ be the space of formal Laurent series in the
indeterminate $z^{-1},$ equipped with the following symplectic
structure: \ben \Omega(\f(z),\g(z)) = {\rm Res}_{z=0}
(\f(-z),\g(z))dz,\quad \f,\g\in \H.  \een Given a sum $\f=\sum
\f_kz^k$, possibly infinite in both directions, we define the vertex
operator \ben e^{\widehat{\f}}:=\exp \Big( \sum_{k\geq 0}
(-1)^{k+1}\sum_{a=1}^N(\f_{-1-k},[\psi_a])q_k^a/\sqrt{\hbar} \Big)
\exp\Big( -\sum_{k\geq 0}\sum_{a=1}^N
(\f_{k},[\psi^a])\sqrt{\hbar}\d_{q_k^a} \Big), \een where $\{ \psi_a
\}$ and $\{ \psi^a \}$, are dual bases in $H$,
$(\psi^a,\psi_b)=\delta^a_b$.  The vertex operator $e^{\widehat{\f}}$
acts on the Fock space of formal series in $q_k^a$ whose coefficients
are formal Laurent series in $\sqrt{\hbar}.$ Applying this
construction to the series
$$
\f^\varphi_\tau(\gl):=\sum_{k\in \Z} I^{(k)}_\varphi(\gl,\tau)(-z)^k,
$$
where $I^{(k)}_\varphi$ are the period vectors \eqref{periods}, we obtain 
vertex operators which will be denoted by $\Gamma_\tau^\varphi(\gl).$

Let $\ga \in H_{2l}(V_{1,0},\Z)$ be a vanishing cycle. Note that the
vertex operator $\Gamma_\tau^\ga(\gl)$ depends on the choice of the
path connecting $(1,0)$ with $(\gl,\tau)$. However, changing this path
corresponds to acting on $\ga$ by a monodromy transformation, which
gives another vanishing cycle $\ga'$. Thus, under this change of path
$\Gamma_\tau^\ga(\gl)$ becomes the vertex operator corresponding
to the vanishing cycle $\ga'$, so that the collection $\{
\Gamma_\tau^\ga(\gl) \}_{\ga \in H_{2l}(V_{1,0},\Z)}$ is
independent of any choices.

\begin{theorem}\label{t2}
Let $\ga,\gb\in H_{2l}(V_{1,0},\Z)$ be two vanishing cycles and $\tau\in
\T$ be an arbitrary point. The following OPE formula holds:
\beq\label{ope_ade} \Gamma^\ga_\tau(\gl) \Gamma_\tau^{\gb}(\mu) = \exp
\Big( \int_{\tau'+(\mu-u') {\bf 1}}^\tau I^{(0)}_\ga(\gl,t)\bullet_t
I^{(0)}_\gb(\mu,t) \Big) : \! \Gamma^\ga_\tau(\gl)
\Gamma_\tau^{\gb}(\mu) \! :\ , \eeq
where $(u',\tau')$ is a generic point in the discriminant. 
\end{theorem}

Here the period $I^{(0)}_\ga(\gl,t)$ should be expanded in a
neighborhood of $\gl=\infty$. Using the residue pairing, we interpret
the integrand in \eqref{ope_ade} as a formal Laurent power series in
$\gl^{-1}$ whose coefficients are multivalued 1-forms on $\T$. The
integration path $C$ is such that the corresponding path $(\mu,t),
t\in C$, does not intersect the discriminant, and the cycle $\gb\in
H_{2l}(V_{\mu,\tau},\Z)$ vanishes when transported to
$H_{2l}(V_{u',\tau'},\Z)$.

\proof It follows from the definition of normal ordering that in order
to prove \eqref{ope_ade} it is enough to show that the integral in the
exponent equals $\Omega(\f_\tau^\ga(\gl)_+,\f_\tau^\gb(\mu)_-),$ i.e.,
we need to show that this expression vanishes for
$\tau=\tau'+(\mu-u'){\bf 1}$ and that its differential equals the
integrand in \eqref{ope_ade}. The first of these two statements is
obvious because $I_\gb^{(-n-1)}(\mu,\tau')=0$ for $n\geq 0$ and
$\mu=u'$. For the differential we have \ben
d\Omega(\f_\tau^\ga(\gl)_+,\f_\tau^\gb(\mu)_-) = \sum_{n\geq 0}
\sum_{a=1}^N (-1)^{n+1}\d_a
\(I^{(n)}_\ga(\gl,\tau),I^{(-n-1)}_\gb(\mu,\tau)\)\, d\tau^a.  \een
Using the Leibniz rule and the differential equations
\eqref{periods:de} we get \ben \sum_{n\geq 0} \sum_{a=1}^N (-1)^n
\Big(\((\d_a\bullet)I_\ga^{(n+1)},I_\gb^{(-n-1)}\) +
\(I_\ga^{(n)},(\d_a\bullet)I_\gb^{(-n)}\)\Big)d\tau^a. \een Applying
the Frobenius property of the multiplication $\bullet$, we find that
this is equal to \ben =\sum_{n\geq 0} \sum_{a=1}^N (-1)^n \Big(
\(I_\ga^{(n+1)}\bullet I_\gb^{(-n-1)},\d_a \) + \(I_\ga^{(n)}\bullet
I_\gb^{(-n)},\d_a \)\Big)d\tau^a.  \een Finally, in the above sum,
with respect to the summation over $n$, all terms cancel except for
$\(I_\ga^{(0)}\bullet I_\gb^{(0)},\d_a \)d\tau^a.$ Formula
\eqref{ope_ade} follows.  \qed

\subsection{Simple singularities and the corresponding Hirota
  equations}

Consider the case of a simple singularity $f$ of type $X_N$, where
$X=ADE$. In this case the Milnor lattice (the middle homology $H_{2l}
(V_{1,0},\Z)$ of the reference Milnor fiber $V_{1,0}$) is identified
with the root lattice of the corresponding simple Lie algebra $\lieg$
of type $X_N$. The invariant inner product $(\cdot\, |\, \cdot)$ on
the root lattice is identified up to the sign $(-1)^l$ with the
intersection pairing on the Milnor lattice.  The monodromy group
acts on the Milnor lattice as the Weyl group. The roots are
identified with vanishing cycles, and the monodromy group action on
them is generated by reflections $\varphi\mapsto
\varphi-\langle\varphi,\alpha\rangle\alpha$ in the vanishing
cycles. The {\em classical monodromy operator} is by definition the
parallel transport of the cycles along the loop $\theta \mapsto
(\gl,\tau)=(\exp i \theta, 0)$ and hence coincides with a Coxeter
element $M$ of the root system.  We keep the notation $\ga_i$ of
Section 2 for the representatives in the orbits of $M$ in the set of
$Nh$ vanishing cycles.

In \cite{GM}, a Hirota bilinear equation was associated to a
simple singularity of type $ADE$ in the following form:
\beqa\label{sing_hirota} && 
\operatorname{Res}_{\gl=\infty}\frac{d\gl}{\gl} \sum_{\ga}
\widetilde{a}_{i} \ \Gamma^{\ga}_0(\gl)\otimes \Gamma^{-\ga}_0(\gl)\
\tau\otimes \tau = \frac{N(h+1)}{12}\ \tau\otimes \tau + \\ \notag
&& \sum_{k=0}^{\infty} \sum_{a=1}^N (m_a+kh)(q_k^a\otimes 1-1\otimes
q_k^a) (\d_{q_k^a}\otimes 1-1\otimes \d_{q_k^a})\ \tau\otimes \tau.  
\eeqa
The sum here is taken over all roots, but the coefficients
$\widetilde{a}_{i}$ are the same for all roots from the $M$-orbit of
$\ga_i$. They are defined by the formula \beq\label{a_alpha}
\widetilde{a}_i:=h \lim_{\ge\rightarrow 0} \exp\Big(- \int_{-{\bf
1}}^{\tau_0-\ge{\bf 1}} I^{(0)}_{\ga_i}(0,t)\bullet
I^{(0)}_{\ga_i}(0,t) - \int_{1}^{\ge}\frac{2d\xi}{\xi} - 4\ln 2\Big).
\eeq The integration here should be understood as follows. We pick a
generic point $\tau_0:=\tau'-u'{\bf 1}$ on the discriminant.  Since
$\ga_i$ is a vanishing cycle, there exists a path from $(1,0)$ to
$(\tau',u')$ such that $\ga_i$ vanishes when transported along it to
the end point. The integration is performed along such a path. For
more details, see \cite{GM}.

It is straightforward to check (see \cite{GM}, Section 8) that the
vertex operators $\Gamma^{\alpha_i}(\zeta)$ of representation theory
given by formula \eqref{vop_kw} turn into the vertex operators
$\Gamma_0^{\alpha_i}(\lambda)$ of singularity theory after the the
following rescaling of the dynamical variables: \beq
\label{rescaling}  
q_k^a=\sqrt{\hbar}\ \prod_{r=0}^k (m_a+rh)y_{m_a+kh}\ 
\eeq
and substitution $\lambda=\zeta^h/h$.

Note that
$$\lim_{\zeta\to w} (1-w^h/\zeta^h)/(1-w/\zeta)=h.$$
Using this, we find from Theorem \ref{t2} (applied at $\tau=0$) and
formulas \eqref{ai_limit1} and \eqref{ope_gamma1} that \beq
\label{tag13} a_i=h \lim_{\lambda\to \mu} (1-\mu/\lambda)^{-2}
\exp\left( - \int_0^{\tau_0+\mu{\mathbf 1}}
I_{\alpha_i}^{(0)}(\lambda,t)\bullet
I_{\alpha_i}^{(0)}(\mu,t)\right).\eeq Taking into account that
$\operatorname{Res} d\gl/\gl = \operatorname{Res} d\zeta/\zeta$, in
order to prove that $a_i=\widetilde{a}_i$ (which is the statement of
the Theorem in the Introduction) it
remains to identify this limit with the one in \eqref{a_alpha}.

\begin{lemma}\label{c1}
We have $a_i=\widetilde{a}_i$ for all $i=1,\ldots,N$.
\end{lemma}

\begin{proof}
Consider the family of integrals \beq \label{tag15} \int_{-{\mathbf
1}}^{\tau_0-\epsilon{\mathbf 1}}
I_{\alpha_i}^{(0)}(\lambda-\mu,t)\bullet I_{\alpha_i}^{(0)}(0,t) +
\int_1^{\lambda-\mu+\epsilon} \frac{2d\xi}{\xi}.\eeq Here $\tau_0$
lies on the discriminant and the first integral is computed along a
discriminant-avoiding path connecting the reference point $-{\mathbf
1}$ with a neighborhood of $\tau_0$ in such a way that the cycle
$\alpha_i$, transported along this path, vanishes as $\epsilon \to
0$. When $\lambda-\mu=0$, the integrals coincide with those in the
exponent of (\ref{a_alpha}), and altogether tend to $-\ln
(\widetilde{a}_i/h) - 4\ln 2$ (up to an integer multiple of $2\pi i$)
in the limit $\epsilon \to 0$.

Now set $\epsilon =0$ in (\ref{tag15}), and write the first integral
as the sum $\int_{-{\mathbf 1}}^{-\mu {\mathbf 1}}+\int_{-\mu {\mathbf
1}}^{\tau_0}$.  The second integral may be converted into the one in
the exponent of (\ref{tag13}) since the integrand is invariant under
the shifts $(\lambda, \mu, t) \mapsto (\lambda + c, \mu+c, t+c{\mathbf
1})$.  The first integral has a well-defined limit as $\mu \to
\lambda$, which is equal to $-\int_1^{\lambda} 2d\xi/\xi$. Indeed, in
this limit, and when $t=\xi {\mathbf 1}$, we have (see \cite{GM}) \beq
\label{tag18} 
I_{\alpha_i}^{(0)}(0,\xi{\mathbf 1})\bullet 
I_{\alpha_i}^{(0)}(0,\xi{\mathbf 1}) = 
-( \alpha_i\, |\, \alpha_i) \rangle \frac{d\xi}{\xi} =-2\frac{d\xi}{\xi}.
\eeq
Note that
\[ \exp \left(\int_1^{\lambda-\mu} \frac{2d\xi}{\xi}-\int_1^{\lambda} 
\frac{2d\xi}{\xi} \right) = (1-\mu/\lambda)^2.\] Thus, at
$\epsilon=0$, the integral (\ref{tag15}) tends in the limit
$\lambda-\mu \to 0$ to $-\ln (a_i/h)$.

Now write the first integral in (\ref{tag15}) as the sum
$\int_{-{\mathbf 1}}^{\tau}+ \int_{\tau}^{\tau_0-\epsilon {\mathbf
1}}$, where $\tau$ is any fixed point along the path. The first
summand depends continuously on $\lambda$, $\mu$ and $\epsilon$ up to
$\lambda=\mu=\epsilon=0$. Thus, in determining if the limits of
(\ref{tag15}) as $\epsilon \to 0$ and $\lambda-\mu \to 0$ commute, we
can replace the base point $-{\mathbf 1}$ by $\tau$, which can be
chosen to lie in a neighborhood of the discriminant point $\tau_0$.

Components of the period vector $I_{\alpha_i}^{(0)}(\lambda,t)$ near a 
non-degenerate critical point where the cycle $\alpha_i$ vanishes
are proportional to 
\[ \frac{1}{\sqrt{\lambda-u}}\left( 1+O(\lambda-u) \right),\]
where $u$ is the critical value, considered as a function of $t$
and taken equal $0$ at $t=\tau_0$. 
Thus, the integrand $I_{\alpha_i}^{(0)}(\lambda-\mu, t)\bullet 
I_{\alpha_i}^{(0)}(0,t)$
with small $\lambda-\mu$ has three types of singularities 
near $t=\tau_0$:
\[ \tag{16} \frac{\sqrt{\lambda-\mu-u}}{\sqrt{-u}}\ du, \ 
    \frac{\sqrt{-u}\ du}{\sqrt{\lambda-\mu-u}},\ \text{and}\ 
    \frac{du}{\sqrt{\lambda-\mu-u}\sqrt{-u}}.\]  
The first singularity is integrable, which makes the order of passing to
the limit irrelevant (indeed, $\int_{-\epsilon}^0$ tends to $0$ as 
$\epsilon \to 0$, uniformly in $\lambda-\mu$). 

The same is true
for the second singularity in (16), which reduces to the first one 
by integration by parts. 

The remaining case literally coincides with the integral (\ref{tag15})
for the $A_1$ singularity $x_1^2/2+x_2x_3+u=\lambda$. In this case the
result may be derived from \cite{GM}. However, for the sake of
completeness we give a direct proof here.

In the $A_1$ case, we have
\[ I^{(0)}_{\alpha}(\lambda,u)=\frac{\pm 2}{\sqrt{2(\lambda-u)}}.\]
Hence (\ref{tag15}) turns into
\[ \tag{17} 
\int_{-1}^{-\epsilon} \frac{4\ du}{\sqrt{2(\lambda-\mu-u)}\sqrt{2(-u)}}
+\int_{1}^{\lambda-\mu+\epsilon} \frac{2d\xi}{\xi}. \]
When $\lambda-\mu=0$, the integrals cancel (modulo $2\pi i {\mathbb Z}$).
On the other hand, when $\epsilon =0$, (17) becomes
\begin{align*} \left. \left. -2\ln \left(-u+\frac{\lambda-\mu}{2}+ 
\sqrt{(\lambda-\mu-u)(-u)}\right) \right|_{-1}^0+ 
2\ln \xi \ \right|_1^{\lambda-\mu} = \\ 
-2\ln\frac{\lambda-\mu}{2} +2\ln \left(1+\frac{\lambda-\mu}{2}+
\sqrt{1+\lambda-\mu}\right)+2\ln (\lambda-\mu).\end{align*}
In the limit $\lambda-\mu\to 0$ this tends to $4\ln 2$ as desired.
This completes the proof.
\end{proof}

This implies the main result of this paper (see the Theorem in the
Introduction): The hierarchy associated in \cite{GM} to a simple
singularity of type $X_N$, where $X=ADE$, coincides with the
Kac--Wakimoto hierarchy for the Lie algebra $X^{(1)}_N$.

\sectionnew{Period realizations of the basic representation}
\label{def}

In the previous section we made use of the vertex operators
$\Gamma^{\ga}_\tau$ defined in terms of the Picard--Lefschetz periods
of a singularity. In this section we will show in the case of simple
singularities how these operators may be used to construct
realizations of the basic representation of the corresponding affine
Kac--Moody algebra. In the special case $\tau=0$ of the unperturbed
singularity, the realization coincides, up to the change of variables
\eqref{rescaling}, with the principal realization we use in Section
2. Hence we obtain a family of deformations of this principal
realization. It is an interesting question whether this family has a
representation theoretic interpretation.

The affine Kac--Moody algebra $\widehat{\lieg}^M$($\cong
\widehat{\lieg}$) of type $X_N$ is spanned by $K$, $d$, $\varphi_{\pm
m},\ m\in E_{+}$, and $A_{\ga,m},\ m\in \Z$, where
\[ \varphi_{m} = \frac{1}{h}\sum_{k=1}^h M^k(\varphi t^m),\ \ \
A_{\ga,m} = \frac{1}{h}\sum_{k=1}^h M^k(A_{\ga} t^m).\] Here
$A_{\ga}\in \lieg_{\ga}$ ($\ga$ is a root), and $\varphi$ is any
element of the Cartan subalgebra $\lieh$. We continue to identify the
Coxeter transformation $M:\lieh \to \lieh$ with the operator of
classical monodromy on the homology group $H_{2l}(V_{1,0},\C)$, which
in its turn is identified by $\varphi \mapsto I^{(0)}_{\varphi}(1,0)$
(a period map composed with the residue pairing operator) with the
local algebra $H$ of the singularity.  We form generating functions
\[ \varphi (\zeta) :=\sum_{m\in E_{+}} \varphi_m \frac{\zeta^{-m}}{-m}+
\varphi_{-m}\frac{\zeta^m}{m},\ \ \ x_{\ga}(\zeta):=\sum_{m\in \Z}
A_{\ga,m} \zeta^{-m}.\]

For each $(\lambda,\tau)\in \C \times \T$, introduce an element of the
Heisenberg algebra $H[[z,z^{-1}]]$
\[ \f^{\varphi}_{\tau}(\lambda):=\sum_{n\in \Z}
I^{(n)}_{\varphi}(\gl,\tau) (-z)^n.\]
To specify the branch of the multivalued period vectors
$I^{(n)}_{\varphi}(\gl,\tau)$, we fix a path avoiding the discriminant
and such that it connects $(\gl,\tau)$ with $(1,0)$ and transports the
cycle $\varphi$ along it.  As functions of $\gl$ near $\gl=\infty$,
the periods expand in Laurent series in $\gl^{1/h}$. In what follows
we put
\[ \zeta = (h \gl)^{1/h}. \] 

Recall that our vertex operators are defined as quantizations of the
Heisenberg group elements $\Gamma_{\tau}^{\ga}(\lambda)=\ :\!
\exp(\widehat{\f}^{\ga}_{\tau}(\lambda) \! :$ (as in Section 3.2),
where $\ga$ is any root.

Introduce {\em phase factors}
\beq\label{phase_factor}
U_{\ga}(\gl,\tau)=
\exp\Big( \frac{1}{2}\int_{-{\bf 1}}^{\tau-\gl {\bf 1}}
I^{(0)}_{\ga}(0,t)\bullet_t I^{(0)}_{\ga}(0,t) 
\Big),\quad \tau\in \T. 
\eeq
The integrand is a covector on $\T$ applied at the point $t$ and
obtained, using the residue pairing identification $T^*_{t}\T \cong
T_{t}\T$, from the product $\bullet_t$ of two (equal) period vectors.
The 1-form consisting of these covectors is integrated along a path
(between the point $\tau-\gl {\mathbf 1}$ and the point $-{\mathbf
1}$).

Finally, introduce the {\em energy operator}
\[ l_{\tau}= z\d_z +\frac{1}{2}-\Theta +z^{-1}E\bullet_\tau \ : \H \to
\H,\] where $E$ is the Euler vector field, and $\Theta$ is the Hodge
grading operator (see Section 3.1).  The energy operator acts on the
symplectic loop space $\H=H((1/z))$ by infinitesimal symplectic
transformation, and so its quantization $\hat{l}_{\tau}$ acts on the
Fock space. We recall that for any infinitesimal symplectic
transformation, $A$, one expresses its quadratic Hamiltonian
$\frac{1}{2}\Omega(A\f,\f)$ in Darboux coordinates $\{ q_k^a,
p_{l,b}\, |\, a,b=1,\dots, N,\ k=0,1,2,\dots \}$ associated with the
polarization $\H = \H_{+}\oplus \H_{-}= T^*\H_{+}$ and then defines
$\widehat{A}$ in the standard way: \beq\label{repr} (q_k^aq_l^b)\sphat
= q_k^aq_l^b/\hbar,\ \ \ (q_k^ap_{l,b})\sphat = (p_{l,b}q_k^a)\sphat =
q_k^a\frac{\d}{\d q_l^b},\ \ \ (p_{k,a}p_{l,b})\sphat = \hbar
\frac{\d^2}{\d q_k^a\d q_l^b}.  \eeq

\begin{theorem}\label{def:basic_rep} 
There exist constants $c_{\ga}$ such that for each $\tau$ the
formulas: \ben && \varphi (\zeta)\ \mapsto \widehat{\bf
f}_{\tau}^{\varphi}(\gl),\\ && x_{\ga} (\zeta) \mapsto c_{\ga}\
\gl^{(\ga|\ga)/2}\ U_{ \ga}(\gl,\tau)\ \Gamma_\tau^{\ga}(\gl), \\ && K
\mapsto 1/h,\quad d\mapsto \widehat{l}_\tau\een define on the Fock
space a representation of the affine Kac--Moody algebra equivalent
to the basic representation of level 1.
\end{theorem}

\proof When $\tau=0$, the formulas coincide (up to the rescaling
\eqref{rescaling}) with the principal realization of the
basic representation of level 1 \cite{KKLW,Kac,KW}.

To derive the theorem for an arbitrary $\tau$, we introduce
intertwining operators $\widehat{S}_\tau$ as quantizations of a
certain symplectic transformations \beq\label{gage}
S_\tau(z)=1+S^{(1)}_{\tau} z^{-1}+S^{(2)}_{\tau}z^{-2}+\cdots,\ \ \
S^{(k)}_{\tau}\in {\rm End}(H),\ \ S_{\tau}^*(-z)S_{\tau}(z)=1.  \eeq
By definition, $\widehat{S}_\tau=\exp^{\widehat{\ln S}_{\tau}}$.

The series $S_\tau$ (also known as {\em calibration} of the
corresponding Frobenius structure) is defined as follows. We introduce
one more period vector, $J(\tau,z)$, corresponding to the {\em complex
oscillating integrals}
\[
(J(\tau,z),\partial_a)= (2\pi z)^{-l-\frac{1}{2}}\ (z\partial_a)\ 
\int e^{f_\tau(x)/z}\omega,\ \ \ 1\leq a\leq N.\]
The cycles of integration here are ``Lefschetz' timbles,''
i.e., relative homology classes of $\C^{2l+1}$ modulo $\left({\rm
Re}(f_\tau/z)\right)^{-1}[-R,-\infty)$ in the limit $R\to \infty$.

When $\omega$ is a primitive volume form, the oscillating integrals
(which are related to the periods $I^{(k)}_{\varphi}(\gl,\tau)$ by
Laplace-like transforms in $\gl$, see for instance \cite{GM}) satisfy
the following system of differential equations:
\[
z\d_a J(\tau,z) = (\d_a\bullet_\tau)J(\tau,z), \ 1\leq a\leq N,\ \ \ 
(z\d_z+E)J=\Theta J,
\]
where $\Theta$ is the Hodge grading operator. Thus, $J$ is a
fundamental solution of a flat connection on the vector bundle with
the fiber $H$ over $(z,\tau)\in (\C-0)\times \T$. The second equation,
which expresses homogeneity properties of oscillating integrals (where
$\deg z=1$), may be rewritten as
\[ \nabla_{\tau} J =0,\ \ \text{where}\ \ 
\nabla_\tau = \d_z + z^{-2}E\bullet_\tau - z^{-1}\Theta.  
\]
One may think of $\nabla_{\tau}$ as an {\em isomonodromic} family of
connection operators over $\C-0$, parametrized by $\tau \in \T$.  The
operators $S_\tau$ are defined as gauge transformations of the form
\eqref{gage} conjugating $\nabla_\tau$ and $\nabla_0 =
\d_z-z^{-1}\Theta$:
\[ \nabla_{\tau} = S_{\tau} \nabla_0 S_{\tau}^{-1}.\]
In the $ADE$ case, one can show that $S_\tau$ satisfying the
initial condition $S_0=1$ exists and is unique. It follows that 
\[
l_{\tau} = S_\tau\, l_0\, S_\tau^{-1},\ \ \mbox{and}\ \ 
z\d_a S_\tau = (\d_a\bullet_\tau)S_\tau,\ 1\leq a\leq N. 
\]
Note that the quantization \eqref{repr} defines a representation of
the Poisson algebra consisting of quadratic Hamiltonians involving
only $pq$ and $q^2$ terms. Both $\widehat{\ln S}_{\tau}$ and
$\widehat{l}_{\tau}$ are obtained by quantizing such
Hamiltonians. Therefore, $\widehat{l}_\tau = \widehat{S}_\tau
\widehat{l}_0\widehat{S}_\tau^{-1}$.

Furthermore, $S_\tau\f_0^\varphi = \f_\tau^\varphi$. Indeed, both
sides satisfy the same differential equation with respect to $\tau$
and the same initial condition at $\tau=0$. Therefore the vertex
operators $\Gamma_\tau^\ga$ and $\Gamma_0^\ga$, being elements of the
Heisenberg group corresponding to the elements $\f_\tau$ and $\f_0$ of
the Heisenberg algebra, are conjugated by $\widehat{S}_{\tau}$ up to
certain scalar factors.  The precise values of the factors have been
found in \cite{GM}, Section 5, Theorem A: \ben
U_{\ga}(\gl,\tau)\Gamma_\tau^\ga(\gl) = \widehat{S}_\tau \
U_{\ga}(\gl,0)\Gamma_0^\ga(\gl)\ \widehat{S}^{-1}_\tau.  \een Thus,
$\widehat{S}_\tau^{-1}$ intertwines the operators described in the
theorem with the operators defining the level 1 basic representation
of the affine Kac--Moody algebra. \qed

\end{document}